\documentclass[11pt]{amsart}
\numberwithin{equation}{section} 

\newtheorem{thm}{Theorem}[section]
\newtheorem{cor}[thm]{Corollary}
\newtheorem{lem}[thm]{Lemma}
\newtheorem{prop}[thm]{Proposition}

\theoremstyle{definition}
\newtheorem{defn}[thm]{Definition}
\theoremstyle{remark}
\newtheorem{rem}[thm]{Remark}

\newcommand{\bea}{\begin{eqnarray}}
\newcommand{\eea}{\end{eqnarray}}
\newcommand{\ba}{\begin{array}}
\newcommand{\ea}{\end{array}}
\newcommand{\bc}{\begin{center}}
\newcommand{\ec}{\end{center}}
\newcommand{\be}{\begin{equation}}
\newcommand{\ee}{\end{equation}}

\def\bn{{\mathbb N}}
\def\s{\sigma}
\def\id{{\bf 1}\!\!{\rm I}}

\def\i{\varepsilon}
\def\t{\tau}

\def\cf{{\mathcal F}}
\def\ch{{\mathcal H}}

\def\w{\omega}

\def\m{\mu}

\def\g{\gamma}
\def\G{\Gamma}

\def\mb{{\mathbf{m}}}
\def\nb{{\mathbf{n}}}
\def\kb{{\mathbf{k}}}
\def\Tb{{\mathbf{T}}}

\normalsize
\begin{document}

\title[Zero-Two Law]
{On a generalized uniform zero-two law for positive contractions of
non-commutative $L_1$-spaces and its vector-valued extension}

\author{Inomjon Ganiev}
\address{Inomjn Ganiev\\
 Department of Computational \& Theoretical Sciences\\
Faculty of Science, International Islamic University Malaysia\\
P.O. Box, 141, 25710, Kuantan\\
Pahang, Malaysia} \email{ {\tt inam@iium.edu.my}}

\author{Farrukh Mukhamedov}
\address{Farrukh Mukhamedov\\
 Department of Computational \& Theoretical Sciences\\
Faculty of Science, International Islamic University Malaysia\\
P.O. Box, 141, 25710, Kuantan\\
Pahang, Malaysia} \email{{\tt far75m@yandex.ru} {\tt
farrukh\_m@iium.edu.my}}

\author{Dilmurod Bekbaev}
\address{Dilmurod Bekbaev\\
 Department of Computational \& Theoretical Sciences\\
Faculty of Science, International Islamic University Malaysia\\
P.O. Box, 141, 25710, Kuantan\\
Pahang, Malaysia}

\begin{abstract}  First, Ornstein and Sucheston proved that for a given positive contraction
$T:L_1\to L_1$ there exists $m\in\bn$ such that
$\big\|T^{m+1}-T^m\|<2$ then
$$
\lim_{n\to\infty}\|T^{n+1}-T^n\|=0.
$$
Such a result was labeled as "zero-two" law. In the present paper,
we prove a generalized uniform "zero-two" law for multi-parametric
family of positive contractions of the non-commutative $L_1$-spaces.
Moreover, we also establish  a vector-valued analogous of the
uniform "zero-two" law for positive contractions of $L_1(M,\Phi)$--
the non-commutative $L_1$-spaces associated with center valued
trace.

\vskip 0.3cm \noindent


\end{abstract}

\maketitle

\section{Introduction}

Let $(X,\cf,\m)$ be a measure space with a positive $\s$-additive
measure $\m$ and let $L_1(X,\cf,\m)$ be the usual associated real
$L_1$-space. A linear operator $T:L_1(X,\cf,\m)\to L_1(X,\cf,\m)$
is called a {\it positive contraction} if  $Tf\geq 0$ whenever
$f\geq 0$ and $\|T\|\leq 1$.

Jamison and Orey \cite{JO} proved that if $P$ is a Markov operator
recurrent in the sense of Harris, with $\s$-finite invariant measure
$\m$, then $\|P^ng\|_1\to 0$ for every $g\in L^1$ with $\int g\
d\m=0$ if (and only if) the chain is aperiodic. Clearly, when the
chain is not aperiodic, taking $f$ with positive and negative parts
supported in different sets of the cyclic decomposition, we have
$\lim_{n\to\infty}\|P^nf\|_1=2\|f\|_1$.

Ornstein and Sucheston \cite{OS} obtained an analytic proof of the
Jamison-Orey result, and in their work they proved the following
theorem \cite[Theorem 1.1]{OS}.

\begin{thm}\label{0-2-s} Let $T:L_11\to
L_1$ be a positive contraction. Then either
\begin{equation}\label{0}
\sup_{\|f\|_1\leq 1}\lim_{n\to\infty}\|T^{n+1}f-T^nf\|=2.
\end{equation}
or $\|T^{n+1}f-T^nf\|\to 0$ for every $f\in L^1$.
\end{thm}

This result was later called a \textit{strong zero-two law}.
Consequently, \cite[Theorem 1.3]{OS}, if $T$ is ergodic with
$T^*\mathbf{1}=\mathbf{1}$ (e.g. $T$ is ergodic and conservative),
then either \eqref{0} holds, or $\|T^ng\|_1\to 0$ for every $g\in
L^1$ with $\int g\ d\m=0$. Some extensions of the strong zero-two
law can be found in \cite{AB,W0}.

Interchanging "sup" and "lim" in the strong zero-two law we have the
following \textit{uniform zero-two law}, proved by Foguel \cite{F}
using ideas of \cite{OS} and \cite{F1}.

 \begin{thm}\label{0-2} Let $T:L_1\to
L_1$ be a positive contraction.  If for some $m\in\bn\cup\{0\}$ one
has $\|T^{m+1}-T^m\|<2$, then
$$
\lim\limits_{n\to\infty}\|T^{n+1}-T^n\|=0.
$$
\end{thm}

Zahoropol \cite{Z} has provided another proof of Theorem \ref{0-2}
which is given in the following theorem.

\begin{thm}\label{B}\cite{Z}  Let $T:L_1\to L_1$ be a positive contraction.
Then for the following statements:
\begin{enumerate}
\item[(i)] there is some $m\in\bn$ such that $\|T^{m+1}-T^m\|<2$;
\item[(ii)] there is some $m\in\bn$ such that
$\|T^{m+1}-(T^{m+1}\wedge T^m)\|<1$;
 \item[(iii)] one has
$$
\lim_{n\to\infty}\|T^{n+1}-T^n\|=0.
$$
\end{enumerate}
the implications hold: (i)$\Rightarrow$(ii)$\Rightarrow$ (iii).
\end{thm}

To establish the implication (ii)$\Rightarrow$(iii) in \cite{Z} the
following auxiliary fact was established.

\begin{thm}\label{Z1}\cite{Z} Let $T,S:L_1\to
L_1$ be two positive contractions such that $T\leq S$. If
$\|S-T\|<1$ then $\|S^n-T^n\|<1$ for all $n\in\bn$.
\end{thm}

In \cite{GMB2015,MTH2010,Mu} we have extended the last result for
several kind of Banach spaces. Therefore, it is natural step is to
find analogous of Theorem \ref{B} in a non-commutative setting. Note
that a strong version of the uniform "zero-two" law was proved in
\cite{MTA2005}, but it was not a desirable result.

The aim of this paper is to prove a non-commutative version of a
generalized uniform "zero-two" law for multi-parametric family of
positive contractions of $L_1$-spaces associated with von Neumann
algebras. As a particular case (when the algebra is commutative), we
recover the results of \cite{Mu,Mu2015}. Moreover, we emphasize that
Theorem \ref{0-2} will be included in the main result as a
particular case.

On the other hand, development of the theory of integration for
measures $\mu$ with the values in Dedekind complete Riesz spaces has
inspired the study of (bo)-complete lattice-normed $L_p$-spaces
(see, for example, \cite{K}). The existence of center-valued traces
on finite von Neumann algebras naturally leads to develop the theory
of integration for this kind of traces. In \cite{GaC}
non-commutative $L_p$-spaces associated with with central-valued
traces have been investigated. Furthermore, in \cite{CZ} more
general $L_p$-spaces associated with Maharam traces have been
studied.

Therefore, another main aim of this paper is to establish the
uniform "zero-two" law for non-commutative $L_1$-spaces associated
with with central-valued traces. It is known \cite{GaC} that
$L_p$-spaces associated with central-valued traces are
Banach-Kantorovich spaces. The theory of Banach-Kantorovich spaces
is now sufficiently well-developed (for instance, see \cite{K}). One
of the important approach to study  Banach-Kantorovich spaces is
provided by the theory of continuous and measurable Banach bundles
\cite{G0,G1}. In this approach, the representation of a
Banach-Kantorovich lattice as a space of measurable sections of a
measurable Banach bundle makes it possible to obtain the needed
properties of the lattice by means of the corresponding stalkwise
verification of the properties. As an application of this approach,
in \cite{GaC} noncommutative $L_p(M,\Phi)$-spaces associated with
center-valued traces are represented as bundle of noncommutative
$L_p$-spaces associated with numerical traces. For other
applications of the mentioned method, we refer the reader to
\cite{AAK,CGa},\cite{GM}-\cite{GMB2015}.

In the second part of this paper, we are going to prove a
vector-valued analogous of the main result for positive contractions
of non-commutative $L_1$-spaces associated with with central-valued
trace. To establish the result, we mainly employ the last mentioned
approach for the existence of vector-valued lifting, which allowed
us to prove the required result. If the algebra is commutative, then
the obtained result extends the main result of \cite{GMB2015}.

Let us outline of the organization of the present paper. In section
2, we collect some necessary well-know facts about non-commutative
$L_1$-spaces. In section 3, we prove an auxiliary result (a
non-commutative analogous of Theorem \ref{Z1}) about dominant
operators. Section 4 is devoted to the proof of a generalized
uniform "zero-two" law for multi-parametric family of positive
contractions of the non-commutative $L_1$-spaces. In section 5, we
recall necessary definitions about $L_1(M,\Phi)$ -- the
non-commutative $L_1$-spaces associated with center valued traces.
Section 6 contains an auxiliary result about the existence of the
non-commutative vector-valued lifting. Finally, in section 7, by
means of the result of section 6, we first prove that every positive
contraction of $L_1(M,\Phi)$ can be represented as a measurable
bundle of positive contractions of non-commutative $L_1$-spaces, and
this allowed us to establish a vector-valued analogous of the
uniform "zero-two" law for positive contractions of $L_1(M,\Phi)$.

\section{Preliminaries}

Throughout the paper  $M$ would be a von Neumann algebra with the
unit $\id$ and let $\t$ be a faithful  normal semi finite trace on
$M$. Therefore we omit this condition from the formulation of
theorems. Recall that an element $x\in M$ is called {\it
self-adjoint} if $x=x^*$. The set of all self-adjoint elements is
denoted by $M_{sa}$. By $M_*$ we denote a pre-dual space to $M$
(see for definitions \cite{BR},\cite{T}).

Let $\frak{N}=\{x\in M: \ \t(|x|)<\infty\}$, here $|x|$ denotes
the modules of an element $x$, i.e. $|x|=\sqrt{x^*x}$.  Define the
map $\|\cdot\|_{1}:\frak{N}\rightarrow [0, \ \infty)$ defined by
the formula $\|x\|_{1}=\tau(|x|)$ is a norm (see \cite{N}). The
completion of $\frak{N}$ with respect to the norm $\|\cdot\|_{1}$
is denoted by $L_{1}(M,\t)$. It is known \cite{N} that the spaces
$L_{1}(M,\tau)$  and $M_*$ are isometrically isomorphic, therefore
they can be identified. Further we will use this fact without
noting.

\begin{thm}\label{2.1}\cite{N} The space $L_{1}(M,\tau)$
coincides with the set
$$
L_{1}=\{ x=\int^{\infty}_{-\infty}\lambda de_{\lambda}\
:\int^{\infty}_{-\infty}|\lambda| d \tau (e_{\lambda})< \infty \}.
$$
Moreover,
$$
\|x\|_{1}=\int^{\infty}_{-\infty}|\lambda |d \tau (e_{\lambda}).
$$
\end{thm}

It is known \cite{N} that the equality
\begin{equation}\label{L11}
L_1(M,\t)=L_1(M_{sa},\t)+iL_1(M_{sa},\t) \end{equation} is valid.
Note that $L_1(M_{sa},\t)$ is a pre-dual to $M_{sa}$.

Let $T:L_1(M,\t)\to L_1(M,\t)$ be any bounded linear operator, by
$\tilde T$ we denote its restriction to $L_1(M_{sa},\t)$. Then due
to \eqref{L11} we have $T(x+iy)= \tilde T(x)+i\tilde T(y)$, where
$x,y\in L_1(M_{sa},\t)$. This means that any linear bounded
operator is uniquely defined by its restriction to
$L_1(M_{sa},\t)$. Therefore, in what follows, we only consider
linear operators on $L_1(M_{sa},\t)$ over real numbers.

Recall that a linear operator $T$ is called {\it positive} if
$Tx\geq 0$ whenever $x\geq 0$. A linear operator $T$ is said to be
a {\it contraction} if $\|T(x)\|_1\leq \|x\|_1$ for all $x\in
L_1(M_{sa},\t)$. Denote
$$\|T\|=\sup\{\|Tx\|_1: \ \|x\|_1=1, x\in
L_1(M_{sa},\t)\}.$$

Let $T,S:L_1\to L_1$ be two positive contractions. In what
follows, we write $T\leq S$ if $S-T$ is a positive operator.

The following auxiliary facts are well known (see for example
\cite{MTH2010}).

\begin{lem}\label{3.1}  Let $T: L_{1}(M_{sa},\tau)\rightarrow
L_{1}(M_{sa},\tau)$ be a positive operator. Then
$$\|T\|=\sup_{\|x\|=1}\|Tx\|=\sup_{\|x\|=1,x\geq
0}\|Tx\|.$$ \end{lem}

\begin{lem}\label{3.2} Let $T, S: L_{1}(M_{sa},\tau)\rightarrow
L_{1}(M_{sa},\tau)$ be two positive contraction such that $T\leq
S$. Then for every $x\in L_1(M_{sa},\t)$, $x\geq 0$ the equality
holds
$$\|Sx-Tx\|=\|Sx\|-\|Tx\|.$$ \end{lem}

\section{Dominant operators}

In this section we are going to prove an auxiliary result related
to dominant operators.

\begin{thm}\label{ZN0} Let
$Z,T,S:L_{1}(M_{sa},\t)\to L_{1}(M_{sa},\t)$ be positive
contractions such that $T\leq S$ and $ZS=SZ$. If there is an
$n_{0}\in\mathbb{N}$ such that $\|Z(S^{n_{0}}-T^{n_{0}})\|<1$.
Then $\|Z(S^{n}-T^{n})\|<1$ for every $n\geq n_{0}.$
\end{thm}

\begin{proof} Assume the contrary, i.e.  $\|Z(S^{n}-T^{n})\|=1$ for some $n>n_{0}.$
Denote
\begin{equation*}
m=\min\{n\in\mathbb{N}:\|Z(S^{n_0+n}-T^{n_0+n})\|=1\}.
\end{equation*}
It is clear that $m\geq 1$. The positivity of $Z$ with $T\leq S$
implies that $Z(S^{n_{0}+n}-T^{n_{0}+n})$ is a positive operator.
Then according to Lemma \ref{3.1} there exists a sequence
$\{x_{n}\}\in L_{1}(M_{sa},\t)$ such that $x_{n}\geq 0$,
$\|x_{n}\|=1, \forall n\in\mathbb{N}$ and
\begin{eqnarray}\label{eq1}
\lim\limits_{n\to\infty}\|Z(S^{n_{0}+m}-T^{n_{0}+m})x_{n}\|=1.
\end{eqnarray}

The positivity of $Z(S^{n_{0}+m}-T^{n_{0}+m})$ and $x_{n}\geq 0$
together with  Lemma \ref{3.2} yield that
\begin{eqnarray}\label{eq2}
\|Z(S^{n_{0}+m}-T^{n_{0}+m})x_{n}\|=\|ZS^{n_{0}+m}x_{n}\|-\|ZT^{n_{0}+m}x_{n}\|
\end{eqnarray}
for every $n\in\bn$. It then follows from \eqref{eq1},\eqref{eq2}
that

\begin{eqnarray}\label{eq3}
&&\lim\limits_{n\to\infty}\|ZS^{n_{0}+m}x_{n}\|=1,\\[2mm]
\label{eq4} &&\lim\limits_{n\to\infty}\|ZT^{n_{0}+m}x_{n}\|=0.
\end{eqnarray}

Thanks to the contractivity of $S$, $Z$ together with $ZS=SZ$ we
obtain
\begin{equation*}
\|ZS^{n_0+m}x_n\|=\|S(ZS^{n_0+m-1}x_n)\|\leq\|ZS^{n_0+m-1}x_n\|\leq
\|S^{m}x_n\|.
\end{equation*}
Hence, the last ones with \eqref{eq3} imply
\begin{equation}\label{ZS1}
\lim\limits_{n\to\infty}\|ZS^{n_0+m-1}x_n\|=1, \ \ \
\lim\limits_{n\to\infty}\|S^{m}x_n\|=1.
\end{equation}

Moreover, the contractivity of $Z,S$ and $T$ ($i=1,2$) implies
that $\|ZT^{n_{0}+m-1}x_{n}\|\leq 1$, $\|T^{m}x_{n}\|\leq 1$ and
$\|ZS^{n_{0}}T^{m}x_{n}\|\leq 1$ for every $n\in\mathbb{N}$.
Therefore, we may choose a subsequence $\{y_{k}\}$ of $\{x_{n}\}$
such that the sequences $\{\|ZT^{n_{0}+m-1}y_{k}\|\}$,
$\{\|T^{m}y_k\|\}$, $\{\|ZS^{n_{0}}T^{m}y_k\|\}$ converge. Hence,
denote their limits as follows
\begin{eqnarray}\label{eq5}
&&\alpha=\lim\limits_{k\to\infty}\|ZT^{n_{0}+m-1}y_{k}\|,\\
\label{eq6}
&&\beta=\lim\limits_{k\to\infty}\|ZS^{n_{0}}T^{m}y_{k}\|,\\
\label{eq7} &&\gamma=\lim\limits_{k\to\infty}\|T^{m}y_{k}\|.
\end{eqnarray}

The inequality $\|Z(S^{n_{0}+m-1}-T^{n_{0}+m-1})\|<1$ with
\eqref{ZS1} implies that $\alpha>0$. Hence  we may choose a
subsequence $\{z_{k}\}$ of $\{y_{k}\}$ such that
$\|ZT^{n_{0}+m-1}z_{k}\|\neq 0$ for all $k\in\mathbb{N}$.

From $\|ZT^{n_{0}+m-1}z_{k}\|\leq \|T^{m}z_{k}\|$ together with
\eqref{eq5}, \eqref{eq7} we find $\alpha\leq\gamma$, and hence
$\gamma>0.$

Now using Lemma \ref{3.2} one gets
\begin{eqnarray}\label{eq10}
\|ZS^{n_{0}}T^{m}z_{k}\|&=&\|ZS^{n_{0}+m}z_k-Z(S^{n_{0}+m}-S^{n_{0}}T^{m})z_{k})\|\nonumber\\
&=&\|ZS^{n_{0}+m}z_{k}\|-\|ZS^{n_0}(S^{m}-T^{m})z_{k}\|\nonumber\\
&\geq&
\|ZS^{n_{0}+m}z_{k}\|-\|S^{m}z_{k}-T^{m}z_{k}\|\nonumber\\
&=&\|ZS^{n_{0}+m}z_{k}\|-\|S^{m}z_{k}\|+\|T^{m}z_{k}\|
\end{eqnarray}
Due to \eqref{eq3} and \eqref{ZS1} we have
$$\lim\limits_{k\to\infty}\big(\|ZS^{n_{0}+m}z_{k}\|-\|S^{m}z_{k}\|\big)=0;$$
which with \eqref{eq10} implies that
$$\lim\limits_{k\to\infty}\|ZS^{n_{0}}T^{m}z_{k}\|\geq\lim\limits_{k\to\infty}\|T^{m}z_{k}\|,$$
this means $\beta\geq\gamma.$

On the other hand, from
$\|ZS^{n_{0}}T^{m}z_{k}\|\leq\|T_2^{m}z_{k}\|$ it follows that
$\gamma\geq\beta$, so $\gamma=\beta$.

Let us denote
$$u_{k}=\frac{T^{m}z_{k}}{\|T^{m}z_{k}\|}, \ \ k\in\bn.$$
Then from $\gamma=\beta$ together with \eqref{eq4}  we obtain
\begin{eqnarray*}
&&\lim\limits_{k\to\infty}\|ZS^{n_{0}}u_{k}\|=
\lim\limits_{k\to\infty}\frac{\|ZS^{n_{0}}T^{m}z_{k}\|}{\|T^{m}z_{k}\|}=1,\\[2mm]
&&\lim\limits_{k\to\infty}\|ZT^{n_{0}}u_{k}\|=
\lim\limits_{k\to\infty}\frac{\|ZT^{n_{0}+m}z_{k}\|}{\|T^{m}z_{k}\|}=0.
\end{eqnarray*}

So, keeping in mind Lemma \eqref{3.2} and the positivity of
$Z(S^{n_{0}}-T^{n_{0}})$, one finds that
$$\lim\limits_{k\to\infty}\|Z(S^{n_{0}}-T^{n_{0}})u_{k}\|=1.$$
Due to  $\|u_{k}\|=1$, $u_{k}\geq 0$ ( for all $k\in\mathbb{N}$)
from Lemma \eqref{3.1} we infer that
$\|Z(S^{n_{0}}-T^{n_{0}})\|=1,$ which is a contradiction. This
completes the proof.
\end{proof}

 We note that the proved theorem extends a main result of the
paper \cite{MTH2010}, which can be seen in the following
corollary.

\begin{cor}\label{ZN1} Let
$T,S:L_{1}(M_{sa},\t)\to L_{1}(M_{sa},\t)$ be positive
contractions such that $T\leq S$. If there is an
$n_{0}\in\mathbb{N}$ such that $\|S^{n_{0}}-T^{n_{0}}\|<1$. Then
$\|S^{n}-T^{n}\|<1$ for every $n\geq n_{0}.$
\end{cor}

The proof immediately follows if one takes $Z=Id$. Note that if
$n_0=1$ and $M$ is a commutative von Neumann algebra, then from
Corollary \ref{ZN1} we immediately get the Zaharopol's result (see
Theorem \ref{Z1}).

\section{A multi-parametric generalization of the zero-two law}

In this section we are going to prove a  multi-parametric
generalization of the zero-two law for positive contractions of
non-commutative $L_1$-space.

Let $T:L_{1}(M_{sa},\t)\to L_{1}(M_{sa},\t)$ be a positive
contraction. Then its conjugate $T^*$ acts on $M_{sa}$, and it is
also positive and enjoys $T^*\id\leq \id$. If one has $T^*\id=\id$,
then $T$ is called \textit{unital positive contraction}.

Let us first introduce some notations. Denote
$\bn_0=\bn\cup\{0\}$. For any
$\mb=(m_1,\dots,m_d),\nb=(n_1,\dots,n_d)\in\bn^d_0$ ($d\geq 1)$ by
the usual way, we define $\mb+\nb=(m_1+n_1,\dots,m_d+n_d)$,
$\ell\nb=(\ell n_1,\dots,\ell n_d)$, where $\ell\in\bn_0$. We
write ${\mathbf{n}}\leq{\mathbf{k}}$ if and only if $n_i\leq k_i$
(i=1,2,\dots,d). We denote $|\nb|:=n_1+\cdots+n_d$.

Let us formulate our main result.

\begin{thm}\label{1-2} Let $Z:L_1(M_{sa},\t)\to L_1(M_{sa},\t)$ be a unital positive contraction. Assume that
$T_k:L_1(M_{sa},\t)\to L_1(M_{sa},\t)$, ($k=1,\dots,d$) be unital
positive contractions such that $ZT_i=T_iZ$, $T_iT_j=T_jT_i$, for
every $i,j\in\{1,\dots,d\}$. If there are $\mb\in\bn_0^d$,
$\kb\in\bn^d_0$ and a positive contraction $S:L_1(M_{sa},\t)\to
L_1(M_{sa},\t)$ such that $SZ=ZS$ with
\begin{eqnarray}\label{T1} &&Z\Tb^{\mb+\kb}\geq
ZS, \ \
Z\Tb^{\mb}\geq ZS \ \ \textrm{with}\\[2mm]\label{T2}
&& \|Z(\Tb^{\mb+\kb}-S)\|<1, \ \ \|Z(\Tb^{\mb}- S)\|<1.
\end{eqnarray}
then for any $\i>0$ there are $M\in\bn$ and $\nb_0\in\bn^d_0$ such
that
$$
\|Z^M(\Tb^{\nb+\kb}-\Tb^{\nb})\|<\i \ \ \ \textrm{for all} \ \
\nb\geq\nb_0.
$$
Here $\Tb^{\nb}:=T_1^{n_1}\cdots T_d^{n_d}$,
$\nb=(n_1,\dots,n_d)\in\bn^d_0$.
\end{thm}

\begin{proof} First we note that for any positive contraction
$\Tb$  on $L_1$-spaces  \cite[p. 310]{Z1} there is $\gamma>0$ such
that
\begin{equation}\label{G-l}
\bigg\|\bigg(\frac{I+\Tb^\kb}{2}\bigg)^\ell-\Tb^\kb\bigg(\frac{I+\Tb^\kb}{2}\bigg)^\ell\bigg\|
\leq\frac{\gamma}{\sqrt{\ell}}.
\end{equation}

Now take any $\i>0$ and fix $\ell_\i\in\bn$ such that
$\gamma/\sqrt{\ell_\i}<\i/2$.

Define
$$
Q_1=\frac{1}{2}(\Tb^{\mb+\kb}-S)+\frac{1}{2}\Tb^\kb(\Tb^\mb-S).
$$
It then follows from \eqref{T1},\eqref{T2} that $ZQ_1$ is positive
and $\|ZQ_1\|<1$. Moreover, one has
$$
\Tb^{\mb+\kb}=\bigg(\frac{I+\Tb^k}{2}\bigg)S+Q_1
$$
where $I$ stands for the identity mapping.

For each $\ell\in\bn$ let us define
$$
Q_{\ell+1}=\bigg(\frac{I+\Tb^\kb}{2}\bigg)^\ell
Q_1S^\ell+\Tb^{\mb+\kb}Q_\ell, \ \ \ell\in\bn.
$$

Taking into account the positivity of $S$ and $Q_1$, one can see
that $Q_\ell$ is a positive operator on $L_1(M_{sa},\t)$ and
$ZQ_\ell=Q_\ell Z$. Moreover, one has
\begin{equation}\label{T3}
\Tb^{\ell(\mb+\kb)}=\bigg(\frac{I+\Tb^\kb}{2}\bigg)^\ell
S^\ell+Q_\ell, \ \ \ell\in\bn.
\end{equation}

Let us prove \eqref{T3} by induction. Clearly, it is valid for
$\ell=1$. Assume that \eqref{T3} is true for $\ell$, and we will
prove it for $\ell+1$. Indeed, one finds
\begin{eqnarray*}
\Tb^{(\ell+1)(\mb+\kb)}&=&\Tb^{m+k}\Tb^{\ell(\mb+\kb)}=\bigg(\frac{I+\Tb^\kb}{2}\bigg)^\ell
\Tb^{\mb+\kb} S^\ell+\Tb^{\mb+\kb} Q_\ell\\[2mm]
&=&\bigg(\frac{I+\Tb^\kb}{2}\bigg)^\ell\bigg(\bigg(\frac{I+\Tb^\kb}{2}\bigg)S+Q_1\bigg)
S^\ell+\Tb^{\mb+\kb}Q_\ell\\[2mm]
&=&\bigg(\frac{I+\Tb^\kb}{2}\bigg)^{\ell+1}
S^{\ell+1}+\bigg(\frac{I+\Tb^\kb}{2}\bigg)^\ell Q_1
S^{\ell}+\Tb^{\mb+\kb}Q_{\ell}\\[2mm]
&=&\bigg(\frac{I+\Tb^\kb}{2}\bigg)^{\ell+1} S^{\ell+1}+Q_{\ell+1}
\end{eqnarray*}
which proves the required equality.

Now let us put $V^{(1)}_\ell=S^\ell$, and
$$
V^{(d+1)}_\ell=\Tb^{\ell(\mb+\kb)}V^{(d)}_\ell+V^{(1)}_\ell
Q^d_\ell, \ \ d\in\bn.
$$

One can see that  for every $d,\ell\in\bn$  the operator
$ZV^{(d)}_\ell$ is positive, since $Z$ and $S$ are commuting.
Moreover, one has
\begin{equation}\label{T4}
\Tb^{d\ell(\mb+\kb)}=\bigg(\frac{I+\Tb^\kb}{2}\bigg)^\ell
V^{(d)}_\ell+Q^d_\ell, \ \ d,\ell\in\bn.
\end{equation}

Again let us prove the last equality by induction.  Keeping in
mind that \eqref{T4} is true for $d$, it is enough to establish
\eqref{T4} for $d+1$. Indeed, we have
\begin{eqnarray*}
\Tb^{(d+1)\ell(\mb+\kb)}&=&\Tb^{\ell(\mb+\kb)}T^{d(\mb+\kb)}=\Tb^{\ell(\mb+\kb)}\bigg(\bigg(\frac{I+\Tb^\kb}{2}\bigg)^\ell
V^{(d)}_\ell+Q^d_\ell\bigg)\\[2mm]
&=&\bigg(\frac{I+\Tb^\kb}{2}\bigg)^\ell \Tb^{\ell(\mb+\kb)}
V^{(d)}_\ell+\bigg(\bigg(\frac{I+\Tb^\kb}{2}\bigg)^\ell
S^\ell+Q_\ell\bigg)Q^d_\ell\\[2mm]
&=&\bigg(\frac{I+\Tb^\kb}{2}\bigg)^{\ell}\bigg(\Tb^{\ell(\mb+\kb)}V^{(d)}_\ell+V^{(1)}_\ell
Q^d_\ell\bigg)+Q^{d+1}_\ell\\[2mm]
&=&\bigg(\frac{I+\Tb^\kb}{2}\bigg)^{\ell}
V^{(d+1)}_\ell+Q^{d+1}_{\ell}
\end{eqnarray*}
which proves \eqref{T4}.

From  $Z^*(\id)=\Tb^*(\id)=\id$, it follows from  \eqref{T4} that
$$
V^{(d)*}_\ell(\id)+Q^{*d}_\ell(\id)=\id.
$$
Now the positivity of $ZV^{(d)}_\ell$ and $ZQ_\ell$ imply that
$\|ZV^{(d)}_\ell\|\leq 1$ and $\|ZQ_\ell\|\leq 1$.

From \eqref{T1} and \eqref{T2}, due to Theorem \ref{ZN0}, one
finds that $\|Z(\Tb^{\ell \mb}-S^\ell)\|<1$ for all $\ell\in\bn$.
Using this inequality with $\Tb^*(\id)=\id$ and the positivity of
$Z(\Tb^{\ell \mb}-S^{\ell})$ we find that
\begin{equation}\label{TSl} \|Z(\Tb^{\ell
\mb}-S^{\ell})\|=\|((\Tb^*)^{\ell
\mb}-S^{*\ell})Z^*\|=\|\id-S^{*\ell}(\id)\|<1.
\end{equation}

The equality \eqref{T3} yields that
$$
Q_\ell^*(\id)=\id-S^{*\ell}(\id).
$$
Hence, from \eqref{TSl}  with the positivity of $ZQ_\ell$ we
obtain
$$
\|ZQ_\ell\|=\|Q_\ell^*(\id)\|=\|\id-S^{*\ell}(\id)\|<1
$$
for all $\ell\in\bn$.

Therefore, there is a number $d_\i\in\bn$ such that
$\|(ZQ_{\ell_\i})^{d_\i}\|<\frac{\i}{4}$. From the commutativity
of $Z$ and $Q_\ell$ one finds
\begin{equation}\label{T7}
\|Z^{d_\i}Q_{\ell_\i}^{d_\i}\|<\frac{\i}{4}.
\end{equation}

Now putting $\nb_0=d_\i\ell_\i(\mb+\kb)$, from \eqref{T4} with
\eqref{T7} we obtain
\begin{eqnarray*}
\|Z^{d_\i}(\Tb^{\nb_0+\kb}-\Tb^{\nb_0})\|&=&\|Z^{d_\i}\big(\Tb^{d_\i\ell_\i(\mb+\kb)+\kb}-\Tb^{d_\i\ell_\i(\mb+\kb)})\|\\[2mm]
&\leq&
\bigg\|Z^{d_\i}\bigg(\Tb^\kb\bigg(\frac{I+\Tb^\kb}{2}\bigg)^{\ell_\i}-\bigg(\frac{I+\Tb^\kb}{2}\bigg)^{\ell_\i}\bigg)\bigg)V^{(d_\i)}_{\ell_\i}\bigg\|\\[2mm]
&&+ \|Z^{d_\i}Q^{d_\i}_{\ell_\i}(\Tb^\kb-I)\|\\[2mm]
&\leq&
\bigg\|\Tb^\kb\bigg(\frac{I+\Tb^\kb}{2}\bigg)^{\ell_\i}-\bigg(\frac{I+\Tb^\kb}{2}\bigg)^{\ell_\i}\bigg\|\\[2mm]
&&+ 2\|Z^{d_\i}Q^{d_\i}_{\ell_\i}\|\\[2mm]
&\leq&\frac{\g}{\sqrt{\ell_\i}}+2\cdot\frac{\i}{4}<\i.
\end{eqnarray*}

Take any $\nb\geq \nb_0$ then from the last inequality one gets
$$
\|\Tb^{\nb+\kb}-\Tb^{\nb}\|=\|\Tb^{\nb-\nb_0}(\Tb^{\nb_0+\kb}-\Tb^{\nb_0})\|\leq
\|\Tb^{\nb_0+\kb}-\Tb^{\nb_0}\|<\i
$$
which completes the proof.
\end{proof}

\begin{cor}\label{1-3} Assume that
$T_k:L_1(M_{sa},\t)\to L_1(M_{sa},\t)$, ($k=1,\dots,d$) be unital
positive contractions such that $T_iT_j=T_jT_i$, for every
$i,j\in\{1,\dots,d\}$. If there are $\mb\in\bn_0^d$,
$\kb\in\bn^d_0$ and a positive contraction $S:L_1(M_{sa},\t)\to
L_1(M_{sa},\t)$ such that
\begin{eqnarray}\label{T1} &&\Tb^{\mb+\kb}\geq
S, \ \
\Tb^{\mb}\geq S \ \ \textrm{with}\\[2mm]\label{T2}
&& \|\Tb^{\mb+\kb}-S\|<1, \ \ \|\Tb^{\mb}- S\|<1.
\end{eqnarray}
then one has
$$
\lim_{\nb\to\infty}\|\Tb^{\nb+\kb}-\Tb^{\nb}\|=0.
$$
\end{cor}

The proof immediately follows from Theorem \ref{1-2} if one takes
$Z=Id$.

\begin{rem} We note
that in \cite{GMB2016,NT} a similar kind of result, for a single
contractions of $C^*$-algebras, has been proved. Our main result
extents it for more general multi-parametric contractions. We point
out that if the algebra becomes commutative, then the proved
theorems cover main results of \cite{Mu2015}.
\end{rem}

\begin{cor}\label{1-4} Let $T,S:L_1(M_{sa},\t)\to L_1(M_{sa},\t)$ be two commuting unital positive contractions.
If for some $m_0\in\bn$ and a positive contraction
$S:L_1(M_{sa},\t)\to L_1(M_{sa},\t)$ one has
\begin{eqnarray*} &&T^{m_0+k}S^{m_0}\geq
S, \ \
T^{m_0}S^{m_0}\geq S \ \ \textrm{with}\\[2mm]\label{T2}
&& \|T^{m_0+k}S^{m_0}-S\|<1, \ \ \|T^{m_0}S^{m_0}- S\|<1.
\end{eqnarray*}
then
$$
\lim_{n,m\to\infty}\|T^{n+k}S^{m}-T^{n}S^{m}\|=0.
$$
\end{cor}

The proof immediately follows from Corollary \ref{1-3} if one takes
$\mb=(m_0,m_0)$ and $\kb=(k,0)$.

\begin{rem} Since the dual of $L_1(M_{sa},\t)$ is $M_{sa}$ then due to the
duality theory the proved Theorem \ref{1-2} holds true if we
replace $L_1$-space with $M_{sa}$.
\end{rem}

\section{Noncommutative $L_1$-space associated with center valued trace}

In this section we recall some necessary notions and facts about the
noncommutative $L_1$-spaces associated with center valued trace.

Let $M$ be any finite von Neumann algebra, $S(M)$ be the set all
measurable operators affiliated to $M$ (see \cite{S} for
definitions). Let $Z$ be some subalgebra of the center $Z(M)$. Then
one may identify $Z$ with $*$-algebra $L_\infty(\Omega,\Sigma,m)$
and do $S(Z)$ with $L_0(\Omega,\Sigma,m)$. Recall that \textit{a
center valued (i.e. $Z$-valued) trace} on the von Neumann algebra
$M$ is a $Z$-linear mapping $\Phi:M\to Z$ with
$\Phi(x^*x)=\Phi(xx^*)\geq 0$ for all $x\in M$. It is clear that
$\Phi(M_+)\subset Z_+$. A trace $\Phi$ is said to be
\textit{faithful} if the equality $\Phi(x^*x)=0$ implies $x=0$,
\textit{normal} if $\Phi(x_{\alpha})\uparrow\Phi(x)$ for every
$x_{\alpha},x\in M_{sa}$, $x_{\alpha}\uparrow x$. Note that the
existence of such kind of traces has been studied in \cite{CZ1}.

Let $M$ be an arbitrary finite von Neumann algebra, $\Phi$ be a
center-valued trace on $M$. The locally measure topology $t(M)$ on
$S(M)$ is  the linear (Hausdorff) topology whose fundamental system
of neighborhoods of $0$ is given by $$ V(B,\varepsilon, \delta ) =
\{x\in S(M)\colon \ \mbox{there exists } \ p\in P(M), z\in P(Z(M))
$$  $$ \mbox{ such that} \ xp\in M, \|xp\|_{M}\leq\varepsilon, \
z^\bot \in W(B,\varepsilon,\delta), \ \Phi_M(zp^\bot)\leq\varepsilon
z\},$$ where $\|\cdot\|_{M}$ is the $C^*$-norm in $M.$  It is known
that $(S(M),t(M))$ is a complete topological $*$-algebra \cite{Y}.

From \cite[\S 3.5]{Mur_m} we have the following criterion for
convergence  in the topology $t(M). $

\begin{prop}\label{2.1.} A net $\{x_\alpha\}_{\alpha\in A} \subset S(M)$ converges to zero in the topology $t(M)$   if and only if $\Phi_M(E^\bot_\lambda (|x_\alpha|) \stackrel{t(M)}{\longrightarrow} 0$ for any $\lambda>0.$
\end{prop}

Following \cite{CZ} an operator  $x\in S(M)$ is said to be {\em
$\Phi$-integrable} if there exists a sequence $\{x_n\} \subset M$
such that $x_n \stackrel{t(M)}{\to} x $ and $\|x_n-x_m\|_\Phi
\stackrel{t(Z)}{\longrightarrow} 0$ as $n,m \to \infty.$

Let $x$ be a $\Phi$-integrable operator from $ S(M).$ Then  there
exists a $\widehat{\Phi}(x)\in S(Z)$ such that $\Phi(x_n)
\stackrel{t(Z)}{\longrightarrow}\widehat{\Phi}(x).$ In addition
$\widehat{\Phi}(x)$ does not depend on the choice of a sequence
$\{x_n\}\subset M,$ for which
$x_n\stackrel{t(M)}{\longrightarrow}x,$ $\Phi(|x_n-x_m|)
\stackrel{t(Z)}{\longrightarrow}0$ \cite{CZ}. It is clear that each
operator $x\in M$ is $\Phi$-integrable and $\widehat{\Phi}
(x)=\Phi(x).$

Denote by $L_1(M,\Phi)$ the set of all $\Phi$-integrable operators
from $S(M).$ If $x\in S(M)$ then $x\in L_1(M,\Phi)$ iff $|x|\in
L_1(M,\Phi),$  in addition $|\widehat{\Phi}(x)|\leq
\widehat{\Phi}(|x|)$ \cite{CZ1}. For any $x\in L_1(M,\Phi),$ set
$\|x\|_{1,\Phi}=\widehat{\Phi}(|x|).$ It is known that $L_1(M,\Phi)$
is a linear subspace of $S(M),$ $ML_1(M,\Phi)M \subset L_1(M,\Phi),$
and $x^*\in L_1(M,\Phi)$ for all $x\in L_1(M,\Phi)$ \cite{CZ1}.

Now let us recall some facts about Banach--Kantorovich spaces over
the ring of measurable functions \cite{G1}.

Let $X$ be a mapping which maps every point $\omega\in \Omega$
    to some Banach space  $(X(\omega),\|\cdot\|_{X(\omega)})$.  In what follows,
    we assume that $X(\omega)\neq \{0\}$  for all
    $\omega\in \Omega.$
A function $u$
  is said to be a \textit{section} of $X$,
   if it is defined almost everywhere in  $\Omega$
    and takes its value $u(\omega)\in X(\omega)$
      for   $\omega\in dom(u),$
      where  $\omega\in dom(u)$ is the domain of
      $u.$ Let   $L$ be some set of sections.

\begin{defn} \cite{G1}. A pair  $(X, L)$ is said to be
a {\it measurable bundle
 of Banach spaces} over $\Omega$   if
\begin{enumerate}
\item[1.]  $\lambda_1 c_1+\lambda_2 c_2\in L$
  for all  $\lambda_1, \lambda_2\in \mathbb{R}$
   and $c_1, c_2\in L,$ where
   $\lambda_1 c_1+\lambda_2 c_2:\omega\in dom(c_1)\cap
   dom(c_2) \rightarrow \lambda_1 c_1(\omega)+\lambda_2 c_2(\omega);$

\item[2.]  the function $||c||:\omega\in  dom(c)\rightarrow
||c(\omega)||_{X(\omega)}$
  is measurable for all $c\in L;$

\item[3.]  for every $\omega\in \Omega$
  the set  $\{c(\omega): c\in L, \omega\in dom(c)\}$
   is dense in $X(\omega).$
\end{enumerate}
   \end{defn}

   A section $s$ is a step-section, if there are pairwise disjoint sets
  $A_1,A_2,\ldots,A_n\in\Sigma$ and sections
 $c_1,c_2,\ldots,c_n\in L$ such that $\bigcup\limits_{i=1}^n
 A_i=\Omega$ è $s(\omega)=\sum\limits_{i=1}^n
 \chi_{A_i}(\omega)c_i(\omega)$ for almost all $\omega\in\Omega$.

 A section $u$ is measurable, if for any $A\in\Sigma$
 there is a sequence $s_n$ of step-sections such that
 $s_n(\omega)\rightarrow u(\omega)$ for almost all $\omega\in A$.

 Let $M(\Omega,X)$ be the set of all measurable sections. By symbol
 $L_0(\Omega,X)$  we denote factorization of the $M(\Omega,X)$ with respect to almost everywhere equality.
 Usually, by $\hat{u}$ we denote a class from $L_0(\Omega,X)$, containing a section $u\in M(\Omega,X)$,
 and by
 $\|\hat{u}\|$ we denote an element of $L_0(\Omega)$,
 containing $\|u(\omega)\|_{X(\omega)}$.
Let ${\mathcal L^{\infty}}(\Omega,X)=\{u\in
M(\Omega,X):\|u(\omega)\|_{X(\omega)}\in \mathcal
L^{\infty}(\Omega)\}$ and $L^{\infty}(\Omega,X)=\{\widehat{u}\in
L_0(\Omega,X): \|\widehat{u}\|\in L^{\infty}(\Omega)\}.$  One can
define the spaces $\mathcal {L^{\infty}}(\Omega,X)$ and
$L^{\infty}(\Omega,X)$ with real-valued norms $\|u\|_{\mathcal
L^{\infty}(\Omega,X)}=\sup\limits_{\omega\in
\Omega}|u(\omega)|_{X(\omega)}$ and
$\|\widehat{u}\|_{\infty}=\bigg\|\|\widehat{u}\|\bigg\|_{L^{\infty}(\Omega)},$
respectively.

 \begin{defn} Let $X,Y$ be measurable bundles of Banach spaces. A set linear operators $\{T(\omega) : X(\omega)\rightarrow
 Y(\omega)\}$ is called \textit{measurable bundle of linear operators} if
 $T(\omega)(u(\omega))$ is measurable section for any measurable
 section $u$.
\end{defn}

Let $(X,L)$ be a measurable bundle of Banach spaces. If each
$X(\omega)$ is a noncommutative $L_1$-space, i.e.
$X(\omega)=L_1(M(\omega),\tau_\omega)$, associated with finite von
Neumann algebras $M(\omega)$ and with strictly normal numerical
trace $\tau_\omega$ on
 $M(\omega)$, then  the measurable bundle  $(X,L)$ of Banach spaces is
 called \textit{ measurable bundle of noncommutative  $L_1$-spaces}.

 \begin{thm} \label{1.2}\cite{GaC} There exists a measurable bundle $(X,L)$
 of noncommutative  $L_1$-spaces $L_1(M(\omega),\tau_\omega)$,
 such that  $L_0(\Omega,X)$ is Banach~---Kantorovich $*$-algebroid,
 which is isometrically and order  $*$-isomorph to  $L_1(M,\Phi)$. Moreover,  the isometric and
  order  $*$-isomorphism $H: L_1(M,\Phi)\rightarrow L_0(\Omega,X)$ can be chosen with the following properties\\
  \begin{enumerate}
\item[(a)] $\Phi(x)(\omega) = \tau_\omega(H(x)(\omega))$ for all
 $x\in M$  and for almost all
 $\omega\in\Omega$;\\
 \item[(b)] $x\in M$ if and only if  $H(x)(\omega)\in
 M(\omega)$ a.e. and there exist positive number
 $\lambda>0$,
 that $\|H(x)(\omega)\|_{M(\omega)}\leq\lambda$ for almost all $\omega$;\\
 \item[(c)] $z\in Z$ if and only if
 $H(z)=(\widehat{z(\omega)\mathbf{1}_\omega)}$ for some
 $\widehat{z(\omega)}\in L_\infty(\Omega)$, where $\mathbf{1}_\omega$~---
 unit of algebra  $M(\omega)$, in particular
 $H(\mathbf{1})(\omega)=
 \mathbf{1}_\omega$ for almost all $\omega$.\\
 \item[(d)] the section $(H(x)(\omega))^*$ is measurable for all  $x\in L_1(M,\Phi)$.\\
 \item[(e)] the section $H(x)(\omega)\cdot H(y)(\omega)$ is measurable for all  $x,y\in M$.
 \end{enumerate}
 \end{thm}

\section{The existence of lifting}

In this section we establishes the existence of the lifting in a
non-commutative setting. Note that in case of $C^*$-algebras, the
existence of the lifting has been given in \cite{GC1} (see also
\cite{GM1}).

Let $M$ be a von Neumann algebra. Then it can be identify with a
linear subspace of $L^\infty(\Omega,X)$ by the isomorphism $H$,
since if $x\in M$, then one has
$$\|H(x)\|_{L_0(\Omega,X)}=\|x\|_1=\Phi(|x|)\in L^\infty(\Omega)$$

\begin{thm} \label{1.3} There exists a mapping  $\ell: M(\subset
L^\infty(\Omega,X))\rightarrow \mathcal{L}^\infty(\Omega,X)$ with
following properties
 \begin{enumerate}
\item[(a)]  for every $x\in M $ one has $\ell(x)\in x,\ {\rm dom}\ \ell (x)=\Omega$;\\
\item[(b)] if $x_1, x_2\in M$ and $\lambda_1, \lambda_2\in
\mathbb{R}$, then
 $\ell(\lambda_1 x_1+ \lambda_2 x_2)=\lambda_1\ell(x_1)+\lambda_2\ell(x_2)$;\\
\item[(c)]
 $\|\ell(x)(\omega)\|_{L_p(M(\omega),\tau_\omega)}=p(\|x\|_p)(\omega)$
 for all $x\in M$ and for all  $\omega\in\Omega$;\\
\item[(d)] if $x\in M,\ \lambda\in
 L^\infty(\Omega)$, then
 $\ell(ex)=p(e)\ell(x)$;\\
\item[(e)] if $x\in M$, then
 $\ell(x^*)=\ell(x)^*$;\\
\item[(f)] if $x,y\in M$, then
 $\ell(xy)=\ell(x)\ell(y)$;\\
 \item[(g)] the set $\{\ell(x)(\omega) : x\in M\}$ is dense
  in $L_p(M(\omega),\tau_\omega)$ for all
 $\omega\in\Omega$.
 \end{enumerate}
 \end{thm}

\begin{proof} Following \cite{GaC} for every $x\in M$ we define
$$
\Phi_0(x)=\Phi(x)(1+\Phi(\id))^{-1}.
$$
One can see that $\Phi_0$ is an $L^{\infty}(\Omega)$-valued faithful
normal trace on $M$. By $\rho$ we denote the lifting on
$L^{\infty}(\Omega)$ (see \cite{G0}). Now define a finite trace
$\varphi_\omega$ on $M$ by
$\varphi_\omega(x)=\rho(\Phi_0(x))(\omega)$, where
$\omega\in\Omega$. Due to \cite[Lemma 6.4.1]{Dik} the function
$s_\omega(x,y)=\varphi_\omega(y^*x)$ is a bi-trace on $M$, and
therefore, the equality $\langle x,y\rangle_\omega=s_\omega(x,y)$
defines a quasi-inner product on $M$.

Denote $I_\omega=\{x\in M: \ s_\omega(x,x)=0\}$. It is known that
$I_\omega$ is a two-sided ideal in $M$, therefore, one considers the
quotient space $\Gamma_\omega=M/I_\omega$, by
$\pi_\omega:M\to\Gamma_\omega$ we denote the canonical mapping. The
involution and multiplication are defined on $\G_\w$ by the usual
way, i.e. $\pi_\w(x)^*=\pi_\w(x^*)$ and
$\pi_\w(x)\cdot\pi_\w(y)=\pi_\w(xy)$. According to \cite[Proposition
6.2.3]{Dik} $\G_\w$ is a Hilbert algebra. By $\ch(\w)$ we denote the
Hilbert space which is the completion of $\G_\w$; the inner product
in $\ch(\w)$ we denote by the same symbol, i.e.
$\langle\cdot,\cdot\rangle_\w$.

The mapping $\pi_\w(x)\to \pi_\w(y)\pi_\w(x)$, $x,y\in M$, can be
extended by continuity to a bounded linear operator $T_\w(y)$ on
$\ch(\w)$. It is known \cite{Dik} that $T_\w(x)$ is a representation
of $M$ in $\ch(\w)$. Let $M(\w)$ be the von Neumann algebra
generated by $T_\w(M)$, i.e. $M(\w)=T_\w(M)''$. By $\m_\w$ we denote
the natural trace on $M(\w)$ which is defined by
$\m_\w(\pi_\w(x))=\langle\pi_\w(x)\id_\w,\id_\w\rangle_\w$ for all
$\pi_\w(x)\in\G_\w$. One can see that $\m_\w$ is a faithful, normal
and finite trace on $M(\w)$ (see \cite[Proposition 6.8.3]{Dik}). Now
let us consider a non-commutative $L_1$-space $L_1(M(\w),\t_w)$,
where $\t_\w(\cdot)=(1+\Phi(\id))(\w)\m_\w(\cdot)$. By
$i_\w:\G_\w\to M(\w)$ one denotes the canonical embedding, and
$j_\w:(M(\w),\t_w)\to L_1(M(\w),\t_\w)$ denotes the natural
embedding. Then $\g_\w=j_\w\circ i_\w\circ\pi_\w$ is a linear
mapping  from $M$ to $L_1(M(\w),\t_\w)$.

Let us define
 $$\ell(x)(\omega)=\gamma_\omega(x)$$ for any $x\in M$.

(a) Since any element $x\in M$ is identified with the element
  $\widehat{\gamma_\omega(x)}$, then one has
 $\ell(x)\in x$ (see \cite{GaC}).

 (b) The linearity  of $\ell$ is obvious.

(c) Let $x\in M$. Then
\begin{eqnarray*}
\|\ell(x)(\omega)\|_{L_1(M(\omega),\tau_\omega)}&=&\|\gamma_\omega(x)\|_{L_1(M(\omega),\tau_\omega)}=\tau_\omega(|\pi_\omega(x)|)\\
&=&\tau_\omega(\pi_\omega(|x|))=\rho(\Phi(|x|))(\omega)\\
&=&\rho(\|x\|_1)(\omega)
\end{eqnarray*}
for all  $\omega\in\Omega$.

(d) Let $\chi_A\in L^\infty(\Omega)$ and $x\in M$, then
 $\chi_A\cdot x\in M$. By $\tilde{\Sigma}$ we denote a
 complete Boolean algebra of equivalent classes w.r.t. a.e. equality, of sets from
 $\Sigma$. The lifting $\rho :
 L^\infty(\Omega)\rightarrow\mathcal{ L}^\infty(\Omega)$ induces a lifting
   $\tilde{\rho} : \tilde{\Sigma}\rightarrow\Sigma$ such that $p(\chi_A)=\chi_{\tilde{p}(A)}$.
 Due to
 \begin{eqnarray*}
 \|\pi_\omega(\chi_A x)\|_{L_1(M(\omega),\tau_\omega)} &=&  \rho(\|\chi_A \cdot
 x\|_1)(\omega)=\rho(\chi_A)(\omega)\cdot
 p(\|x\|_1)(\omega)\\
 &=& \rho(\chi_A)(\omega)\cdot\|\pi_\omega(x)\|_{L_1(M(\omega),\tau_\omega)}\\
 &=&
 \chi_{\tilde{p}(A)}(\omega)\|\pi_\omega(x)\|_{L_1(M(\omega),\tau_\omega)},
 \end{eqnarray*} we obtain
 $\pi_\omega(\chi_A\cdot x)=0$, if $\omega\in\tilde{\rho}(A)$. Let
 $\omega\in\tilde{\rho}(A)$, then
\begin{eqnarray*}
 \|\pi_\omega(\chi_A\cdot
 x)-\pi_\omega(x)\|_{L_1(M(\omega),\tau_\omega)}&=& \|\pi_\omega(\chi_{\Omega\setminus
 A}\cdot x)\|_{L_1(M(\omega),\tau_\omega)}\\
 & =&\rho(\|\chi_{\Omega\setminus A}\cdot
 x\|_1)(\omega)= \rho(\chi_{\Omega\setminus A})(\omega)\cdot
 p(\|x\|_1)(\omega)\\
 &=& \chi_{\tilde{\rho}(\Omega\setminus
 A)}(\omega)\cdot \|\pi_\omega(x)\|_{L_1(M(\omega),\tau_\omega)}=0.
 \end{eqnarray*}  Therefore, we have $\pi_\omega(\chi_A \cdot x)= \chi_{\tilde{\rho}(A)}(\omega)
 \cdot\pi_\omega(x)= p(\chi_A)(\omega)\cdot\pi_\omega(x)$ for all
 $\omega\in\Omega$.

 Let $\lambda=\sum\limits_{i=1}^n r_i\chi_{A_i}\in
 L^\infty(\Omega)$ be a simple function. Then
 \begin{eqnarray*}\pi_\omega(\lambda
 x)&=& \pi_\omega\bigg(\sum\limits_{i=1}^n r_i \chi_{A_i} x\bigg)=
 \sum\limits_{i=1}^n \pi_\omega(r_i \chi_{A_i} x)\\
 &=&
 \sum\limits_{i=1}^n r_i \rho(\chi_{A_i})(\omega)\pi_\omega(x) =
\rho(\lambda)(\omega) \pi_\omega(u).
 \end{eqnarray*}

The density argument implies that for any $\lambda\in
L^\infty(\Omega)$ there exists a sequence of simple functions
  $\{\lambda_n\}$ such that
 $\|\lambda_n-\lambda\|_{L^\infty(\Omega)}\rightarrow0$ as $n\rightarrow\infty$.
 From
 \begin{eqnarray*}
 \|\pi_\omega(\lambda_n x)-\pi_\omega(\lambda
 x)\|_{L_1(M(\omega),\tau_\omega)}&=& \|\pi_\omega((\lambda_n - \lambda)x)\|_{L_1(M(\omega),\tau_\omega)} =
\rho(\|(\lambda_n-\lambda)x\|_1)(\omega)\\[2mm]
 &=& \rho(|\lambda_n -
 \lambda|)(\omega)\cdot \rho(\|x\|_1)(\omega)\\[2mm]
 &\leq&
 \|\rho(\lambda_n - \lambda)\|_{\mathcal{ L}^\infty(\Omega)}\cdot
 \|\rho(\|x\|_1)\|_{\mathcal{ L}^\infty(\Omega)}\\[2mm]
 &=&
 \|\lambda_n-\lambda\|_{L^\infty(\Omega)}\cdot \|\|x\|_1 \|_{L^\infty(\Omega)}
 \end{eqnarray*}
 one gets
 $\pi_\omega(\lambda
 x)=\lim\limits_{n\rightarrow\infty}\pi_\omega(\lambda_n x)$ for all
 $\omega\in\Omega$. So,
 $$\pi_\omega(\lambda
 x)=\lim\limits_{n\to\infty}\pi_\omega(\lambda_n x)=
 \lim\limits_{n\rightarrow\infty} \rho(\lambda_n)(\omega)\pi_\omega(x)=
\rho(\lambda)(\omega)\pi_\omega(u)$$
 for all $\omega\in\Omega$.

 Hence,
 $$i_\omega
 (\pi_\omega(\lambda x)=
 i_\omega(\rho(\lambda)(\omega)\pi_\omega(x))=
 \rho(\lambda)(\omega)i_\omega(\pi_\omega(x))= \rho(\lambda)(\omega)
 i_\omega (\pi_\omega(x))$$ and $j_\omega(\lambda x))=\rho(\lambda)(\omega)
 j_\omega (x)$ for all $\omega\in\Omega$. These mean that  $\ell(\lambda x)=\rho(\lambda)\ell(x)$
 for any $x\in M$ and $\lambda\in
 L^\infty(\Omega)$.

(e) According to  $\gamma_\omega(x^*)=\gamma_\omega(x)^*$ for any
$x\in M$ we get $\ell(x^*)=\ell(x)^*$.

 (f) From
$\gamma_\omega(xy)=\gamma_\omega(x)\gamma_\omega(y)$ for any $x,y\in
M$ it follows that  $\ell(xy)=\ell(x)\ell(y)$ for any $x,y\in M$.

(g) By the construction of $\gamma_\omega$ the set
$\{\gamma_\omega(x): x\in M\}$ is dense in
$L_1(M(\omega),\tau_\omega)$ for any $\omega\in\Omega$. Therefore,
the set $\{\ell(x)(\omega) : x\in M\}$ is dense
  in $L_1(M(\omega),\tau_\omega)$ for all
 $\omega\in\Omega$.

The proof is complete. \end{proof}

\begin{defn}
The defined map $\ell$ in Theorem \ref{1.3} is called \textit{a
noncommutative vector-valued lifting associated with the lifting
$\rho$}.
\end{defn}

\section{Vector valued analogous of the noncommutative zero-two law}

In this section, we are going to prove a vector-valued analogous of
Theorem \ref{1-2}.

 Let $M$ be any finite von Neumann algebra, and $L_1(M,\Phi)$ be the noncommutative $L_1$-space, associated with $M$ and the center valued trace  $\Phi$.
 Let   $(X,L)$ be a measurable bundle of noncommutative  $L_1$-spaces
$L_1(M(\omega),\tau_\omega)$, associated with finite  von Neumann
algebras $M(\omega)$ and with strictly normal numerical traces
$\tau_\omega$ on
 $M(\omega)$, corresponding to $L_1(M,\Phi)$.

 \begin{thm}\label{Tw-1}
 Let $T:L_1(M,\Phi)\rightarrow L_1(M,\Phi)$ be a
 positive contraction  with $T(\mathbf{1})\leq\mathbf{1}$. Then
 there exists a measurable bundle of positive contractions
 $T_\omega:L_1(M(\omega),\tau_\omega)\rightarrow
 L_1(M(\omega),\tau_\omega)$ such that
 $$T_\omega(x(\omega))=(Tx)(\omega)$$ for all
 $x\in L_1(M,\Phi)$ and for almost all
 $\omega\in\Omega$, and
 $$\|T\|(\omega)=\|T_\omega\|_{L_1(M(\omega),\tau_\omega)\rightarrow
 L_1(M(\omega)}$$
\end{thm}

 \begin{proof} Let $x\in M_{sa}$. Then
 $$
 |Tx|\leq T(|x|)\leq \|x\|_M T(\mathbf{1})\leq \|x\|_M
 \mathbf{1},
 $$ i.e. $Tx\in M.$  If $x\in M,$ there exist $y,z\in M_{sa}$ such
 that $x=y+iz.$ Then $Tx=Ty+iTz.$ As $Ty, Tz\in M$ we have $Tx\in M.$

 Let $\ell: M(\subset L^\infty(\Omega,X))\rightarrow
\mathcal{L}^\infty(\Omega,X)$ be the noncommutative vector-valued
lifting associated with lifting $p$.

We define the linear operator $\varphi_\omega$ from
$\{\ell(x)(\omega): x \in M\}$ into $L_1(M(\omega),\tau_\omega)$
by
$$\varphi_\omega(\ell(x)(\omega))=\ell(Tx)(\omega)$$
The contractivity of $T$ implies that
\begin{eqnarray*}
\|\varphi_\omega(\ell(x)(\omega))\|_{L_1(M(\omega),\tau_\omega)}&=&\|\ell(Tx)(\omega)\|_{L_1(M(\omega),\tau_\omega)}=\rho(\|Tx\|_1)(\omega)\\[2mm]
&\leq&
\rho(\|x\|_1)(\omega)=\|\ell(x)(\omega)\|_{L_1(M(\omega),\tau_\omega)}.
\end{eqnarray*}
This means that $\varphi_\omega$ is bounded and well defined.
Moreover, one has
$$\|\varphi_\omega\|_{L_1(M(\omega),\tau_\omega)\rightarrow
 L_1(M(\omega,\tau_\omega)}\leq 1.$$ The positivity of $T$ yields that $\varphi_\omega$  is positive as well.

Since the set $\{\ell(x)(\omega): x \in M\}$ is dense in
$L_1(M(\omega),\tau_\omega)$, then one can extend  $\varphi_\omega$
by continuity to a linear positive contraction
$T_\omega:L_1(M(\omega),\tau_\omega)\rightarrow
 L_1(M(\omega),\tau_\omega)$ by $T_\omega(x(\omega))=\lim\limits_{n\rightarrow\infty}\varphi_\omega(\ell(x_n)(\omega)).$

From $\varphi_{\omega}(\ell(x)(\omega))\in \mathcal
L^{\infty}(\Omega,X),$ for any $x\in M$ we obtain
$T_{\omega}(x(\omega))\in M(\Omega,X)$ for any  $x\in M(\Omega,X).$
Therefore, $\{T_{\omega}\}$ is a measurable bundle of positive
operators.

Using the same argument as in the proof Theorem 4.5 \cite{GM1} one
can prove
 $$T_\omega(x(\omega))=(Tx)(\omega)$$ for all
 $x\in L_1(M,\Phi)$ and for almost all
 $\omega\in\Omega$.

 Now let us establish $\|T\|(\omega)=\|T_\omega\|_{L_1(M(\omega),\tau_\omega)\rightarrow
 L_1(M(\omega)}$.

 Let $x\in M$. Then
\begin{eqnarray*}
\|\varphi_\omega(\ell(x)(\omega))\|_{L_1(M(\omega),\tau_\omega)}&=&\|\ell(Tx)(\omega)\|_{L_1(M(\omega),\tau_\omega)}=\rho(\|Tx\|_1)(\omega)\\[2mm]
&\leq&\rho(\|T\|\|x\|_1)(\omega)=\rho(\|T\|)(\omega)p(\|x\|_1)(\omega)\\[2mm]
&=&\rho(\|T\|)(\omega)\|\ell(x)(\omega)\|_{L_1(M(\omega),\tau_\omega)}.
\end{eqnarray*}

If $x(\omega)\in L_1(M(\omega),\tau_\omega)$, then one finds
\begin{eqnarray*}
\|T_\omega
x(\omega))\|_{L_1(M(\omega),\tau_\omega)}&=&\lim\limits_{n\rightarrow\infty}\|\varphi_\omega(\ell(x_n)(\omega))\|_{L_1(M(\omega),\tau_\omega)}\\[2mm]
&\leq&\rho(\|T\|)(\omega)\lim\limits_{n\rightarrow\infty}\|\ell(x_n)(\omega)\|_{L_1(M(\omega),\tau_\omega)}\\[2mm]
&=&\rho(\|T\|)(\omega)\|x(\omega)\|_{L_1(M(\omega),\tau_\omega)}.
\end{eqnarray*}

Hence, $\|T_\omega\|_{L_1(M(\omega),\tau_\omega)\rightarrow
 L_1(M(\omega)}\leq
\rho(\|T\|)(\omega).$

By  \cite[Proposition 2]{GK} for any $\varepsilon>0$ there exists
$x\in L_1(M,\Phi)$ with $\|x\|_1=\mathbf{1}$ such that
$$\|Tx\|_1\geq \|T\|-\varepsilon\mathbf{1}.$$
Then
\begin{eqnarray*}
\rho(\|T\|)(\omega)-\varepsilon&\leq&
\rho(\|Tx\|_1)(\omega)=\|\ell(Tx)(\omega)\|_{L_1(M(\omega),\tau_\omega)}\\[2mm]
&=&\|T_\omega\ell(x)(\omega)\|_{L_1(M(\omega),\tau_\omega)}\\[2mm]
&\leq& \|T_\omega\|_{L_1(M(\omega),\tau_\omega)\rightarrow
 L_1(M(\omega,\tau_\omega)}\|\ell(x)(\omega)\|_{L_1(M(\omega),\tau_\omega)}\\[2mm]
 &=&\|T_\omega\|_{L_1(M(\omega),\tau_\omega)\rightarrow
 L_1(M(\omega)}p(\|x\|_1)(\omega)\\[2mm]
 &=&\|T_\omega\|_{L_1(M(\omega),\tau_\omega)\rightarrow
 L_1(M(\omega,\tau_\omega)}.
 \end{eqnarray*}

The arbitrariness of  $\varepsilon$ yields

 $$p(\|T\|)(\omega)\leq \|T_\omega\|_{L_1(M(\omega),\tau_\omega)\rightarrow
 L_1(M(\omega,\tau_\omega)}.$$

 Hence $$p(\|T\|)(\omega)= \|T_\omega\|_{L_1(M(\omega),\tau_\omega)\rightarrow
 L_1(M(\omega,\tau_\omega)}$$ for all $\omega\in\Omega$  or
 equivalently we have
$$\|T\|(\omega)= \|T_\omega\|_{L_1(M(\omega),\tau_\omega)\rightarrow
 L_1(M(\omega,\tau_\omega)}$$
for almost all
 $\omega\in\Omega$. This completes the proof.
\end{proof}

\begin{cor}\label{Tw-2} Let $T:L_1(M,\Phi)\rightarrow L_1(M,\Phi)$ be
 a positive contraction with $T(\mathbf{1})=\mathbf{1}$. Then
 there exists a measurable bundle of positive contractions
 $T_\omega:L_1(M(\omega),\tau_\omega)\rightarrow
 L_1(M(\omega),\tau_\omega)$ such that
 $$T_\omega(x(\omega))=(Tx)(\omega)$$ for all
 $x\in L_1(M,\Phi)$ and for almost all
 $\omega\in\Omega$ and
 $$\|T\|(\omega)=\|T_\omega\|_{L_1(M(\omega),\tau_\omega)\rightarrow
 L_1(M(\omega)}$$
\end{cor}

Let $T:L_1(M,\Phi)\rightarrow L_1(M,\Phi)$ be a positive contraction
with $T(\mathbf{1})=\mathbf{1}$ and
$T_\omega:L_1(M(\omega),\tau_\omega)\rightarrow
 L_1(M(\omega),\tau_\omega)$  ba a measurable bundle of positive
 contractions. Then $T$ is called \textit{unital positive contraction},
 if one has
 $T_\omega^*(\mathbf{1}_\omega)=\mathbf{1}_\omega$ for almost all
 $\omega\in\Omega$.

\begin{thm} Assume that $T:L_1(M,\Phi)\rightarrow L_1(M,\Phi)$
be a unital positive contraction. If there are $m,k\in
\mathbb{N}_0=\mathbb{N}\cup\{0\}$ and a positive contraction
$S:L_1(M,\Phi)\rightarrow L_1(M,\Phi)$ with
$S(\mathbf{1})=\mathbf{1}$ such that
$$T^{m+k}\geq S,  T^m\geq S \ \ \textrm{with}$$
$$\|T^{m+k}-S\|<\mathbf{1},  \|T^{m}-S\|<\mathbf{1}$$
then
$$(o)-\lim\limits_{n\rightarrow\infty}\|T^{n+k}-T^n\|=0.$$
\end{thm}

\begin{proof}
By Corollary \ref{Tw-2} there exists
$T_\omega:L_1(M(\omega),\tau_\omega)\rightarrow
 L_1(M(\omega,\tau_\omega)$ and $S_\omega:L_1(M(\omega),\tau_\omega)\rightarrow
 L_1(M(\omega,\tau_\omega)$ such that $T_\omega(x(\omega))=(Tx)(\omega)$ and $S_\omega(x(\omega))=(Sx)(\omega)$ for all
 $x\in L_1(M,\Phi)$ and for almost all
 $\omega\in\Omega$.

 From  $T^{m+k}\geq S,  T^m\geq S$ we get $T_\omega^{m+k}\geq S_\omega,  T_\omega^m\geq
 S_\omega$ for almost all
 $\omega\in\Omega$. Since $\|T^{m+k}-S\|<\mathbf{1},
 \|T^{m}-S\|<\mathbf{1}$ one finds  $$\|T_\omega^{m+k}-S_\omega\|_{L_1(M(\omega),\tau_\omega)\rightarrow
 L_1(M(\omega),\tau_\omega)}<1,  \ \  \|T_\omega^{m}-S_\omega\|_{L_1(M(\omega),\tau_\omega)\rightarrow
 L_1(M(\omega),\tau_\omega)}<1$$ for almost all
 $\omega\in\Omega$. Then using $T_\omega^*(\mathbf{1}_\omega)=\mathbf{1}_\omega$
 we obtain that the measurable bundle of positive contractions
 $T_\omega$ satisfies all conditions Corollary \ref{1-3}  for almost all
 $\omega\in\Omega$. Therefore $$\lim\limits_{n\rightarrow\infty}\|T_\omega^{n+k}-T^n_\omega\|_{L_1(M(\omega),\tau_\omega)\rightarrow
 L_1(M(\omega),\tau_\omega)}=0$$ for almost all
 $\omega\in\Omega$.

 According to
 $$\|T^{n+k}-T^n\|(\omega)=\|T_\omega^{n+k}-T^n_\omega\|_{L_1(M(\omega),\tau_\omega)\rightarrow
 L_1(M(\omega),\tau_\omega)}, \ \ \textrm{a.e.}
 $$ we obtain
 $\lim\limits_{n\rightarrow\infty}\|T^{n+k}-T^n\|(\omega)=0$ for almost all
 $\omega\in\Omega$, which means
$$(o)-\lim\limits_{n\rightarrow\infty}\|T^{n+k}-T^n\|=0.$$
\end{proof}

\section*{Acknowledgement} The first  author  acknowledges  the MOE Grant FRGS13-071-0312.

\end{document}